\documentclass[twoside,12pt]{article}
\usepackage{amsmath}
\usepackage{amssymb}
\usepackage{graphicx}
\usepackage{epigraph}
\usepackage{amsfonts,amssymb,color}
\usepackage[mathscr]{eucal}
\usepackage{amsmath, amsthm}
\usepackage{mathrsfs}
\usepackage{amsbsy}
\usepackage{float}
\usepackage{verbatim}
\usepackage{amsfonts}
\usepackage{pifont}
\usepackage{subfig}
\usepackage{hyperref}
\usepackage{epigraph}
\usepackage{amsmath, amsthm, amscd}
\usepackage{geometry}
\usepackage[shortlabels]{enumitem}
\usepackage{tikz}
\usetikzlibrary{shapes.geometric, arrows, positioning}

\usepackage{lipsum}
\usepackage{fancyhdr}
\fancypagestyle{mypagestyle}{%
  \fancyhf{}
  \fancyhead[OC]{\uppercase{Sequential convergence on the space of Borel measures}}
  \fancyhead[EC]{\uppercase{{Liangang Ma}}}
  \fancyfoot[C]{\thepage}%
}
\usepackage{titlesec}
\titleformat*{\section}{\bfseries\fontsize{13}{20}\selectfont}
\titlelabel{\thetitle.\quad}

\newtheorem{theorem}{Theorem}[section]
\newtheorem*{Vitali-Hahn-Saks Theorem}{Vitali-Hahn-Saks Theorem}
\newtheorem*{Definition 2.6'}{Definition 2.6'}

\newtheorem{proposition}[theorem]{Proposition}
\newtheorem{corollary}[theorem]{Corollary}
\newtheorem{lemma}[theorem]{Lemma}
\newtheorem{definition}[theorem]{Definition}
\newtheorem{remark}[theorem]{Remark}
\newtheorem{example}[theorem]{Example}

\numberwithin{equation}{section}

\newcommand\blfootnote[1]{%
  \begingroup
  \renewcommand\thefootnote{}\footnote{#1}%
  \addtocounter{footnote}{-1}%
  \endgroup
}

\title{\large \textbf{\uppercase{Sequential convergence on the space of Borel measures}}}                  
\author{\uppercase{Liangang Ma}}
\date{}

\begin{document}

\maketitle

\blfootnote{The work is supported by ZR2019QA003 from SPNSF and 12001056 from NSFC.}

\begin{abstract}
 We study equivalent descriptions of the vague, weak, setwise and total-variation (TV) convergence of sequences of Borel measures on metrizable and non-metrizable topological spaces in this work. On metrizable spaces, we give some equivalent conditions on the vague convergence of sequences of measures following Kallenberg, and some equivalent conditions on the TV convergence of sequences of measures following Feinberg-Kasyanov-Zgurovsky. There is usually some hierarchy structure on the equivalent descriptions of convergence in different modes, but not always. On non-metrizable spaces, we give examples to show that these conditions are seldom enough to guarantee any convergence of sequences of measures. There are some remarks on the attainability of the TV distance and more modes of sequential convergence at the end of the work.

\end{abstract}

\section{Introduction}

Let $X$ be a topological space with its Borel $\sigma$-algebra $\mathscr{B}$. Consider the collection $\mathcal{\tilde{M}}(X)$ of all the Borel measures on $(X, \mathscr{B})$. When we consider the regularity of some mapping
\begin{center}
$f: \mathcal{\tilde{M}}(X)\rightarrow Y$
\end{center}
with $Y$ being a topological space, some topology or even metric is necessary on the space  $\mathcal{\tilde{M}}(X)$ of Borel measures. Various notions of topology and metric grow out of different situations on the space  $\mathcal{\tilde{M}}(X)$ in due course to deal with the corresponding concerns of regularity. In those topology and metric endowed on $\mathcal{\tilde{M}}(X)$, it has been recognized that the vague, weak, setwise topology as well as the total-variation (TV) metric are highlighted notions on the topological and metric description of $\mathcal{\tilde{M}}(X)$ in various circumstances, refer to \cite{Kal, GR, Wul}. From the viewpoint of sequential convergence (\cite{Cla}) in $\mathcal{\tilde{M}}(X)$, they induce corresponding notions of convergence of  sequences of Borel measures. One is recommended to refer to \cite{Fol, Kallen2, Kallen3, Kle} for vague convergence, to \cite{Bil1, Bil2, HL1, Kallen4, Kallen5, Kle, Las2} for weak convergence, to \cite{FKL, FKZ1, FKZ2, Las1, Las3} for setwise and TV convergence of sequences of measures in $\mathcal{\tilde{M}}(X)$. The focus of this work is on the sequential convergence of measures instead of the topology on $\mathcal{\tilde{M}}(X)$, however, they are closely related ways on describing regularity of $\mathcal{\tilde{M}}(X)$.

In this work we mean to study the necessary or sufficient conditions on the vague, weak, setwise as well as the TV convergence of sequences of  measures in $\mathcal{\tilde{M}}(X)$. We tend to give equivalent conditions (that is, necessary and sufficient) on the sequential convergence in these modes, however, some conditions being merely necessary or sufficient seem to be more interesting in due courses. We collect equivalent (or merely necessary or sufficient) descriptions of the vague, weak, setwise as well as the total-variation (TV) convergence of sequences of  measures already established, while we formulate new descriptions for further interests and study relationships between these descriptions.     

Note that until now most results on vague, weak, setwise and TV convergence of sequences of measures are set on $\mathcal{\tilde{M}}(X)$ with the ambient space $X$ being metrizable. This is reasonable from the viewpoint of actual applications of these results in due courses. See for example the application of vague convergence on non-interactive particle systems by O. Kallenberg on Euclidean spaces in \cite{Kallen1}, application of weak convergence on interactive particle systems by J. T. Cox, A. Klenke and E. A. Perkins on locally compact Polish spaces in \cite{CK, CKP}, application of the weak, setwise or TV convergence in the Markov decision processes by  E. Feinberg, P. Kasyanov and M. Zgurovsky \cite{FK, FKZ1, FKZ3, FKZ4}, in the Markov chains by O. Hern\'andez-Lerma and J. Lasserre \cite{HL1, HL3} on metric spaces, application of the setwise convergence in iterated function systems by Ma \cite{Ma1, Ma2} on metric spaces.  However, it is interesting to ask what will happen when the ambient space $X$ losts metrizability (although these questions are not always well posed). We provide examples to show that popular necessary or sufficient descriptions are seldom true for any sequential convergence of measures on non-metrizable spaces.

Another concern for us is on the scope of measures to which these results apply. Let 
\begin{center}
$\mathcal{\hat{M}}(X)=\{\nu\in\mathcal{\tilde{M}}(X): \nu(X)<\infty\}$
\end{center}   
be the collection of all the finite Borel measures on $X$ and
\begin{center}
$\mathcal{M}(X)=\{\nu\in\mathcal{\hat{M}}(X): \nu(X)=1\}$
\end{center}   
be the collection of all the Borel probability measures on $X$. We mainly confine our attention to finite measures in $\mathcal{\hat{M}}(X)$ in this work, while many results on it are shared on $\mathcal{M}(X)$. Some results extend onto $\mathcal{\tilde{M}}(X)$, but not always.    

The organization of the paper is as following. In Section \ref{sec2} we introduce notions of the vague, weak, setwise and TV convergence of sequences of measures in $\mathcal{\tilde{M}}(X)$, between which we present results on the equivalent descriptions of the vague and TV convergence-Theorem \ref{thm6} and \ref{thm3}. Based on Theorem \ref{thm5} in this section one can construct examples that equivalent descriptions of these convergences fail. Section \ref{sec3} to \ref{sec6} are devoted to (equivalent, necessary or sufficient) descriptions of the vague, setwise and TV convergence of sequences of measures in $\mathcal{\hat{M}}(X)$ (some results or examples are on $\mathcal{\tilde{M}}(X)$ or $\mathcal{M}(X)$) respectively. Cases that the ambient space $X$ being metrizable or non-metrizable are both in our consideration throughout the work. In the last section we first point out some distinctions on the equivalent descriptions of the TV metric by their attainability, then give some heuristic definitions on more modes of sequential convergence of measures.

\section{The vague, weak, setwise and TV convergence of sequences of measures in $\mathcal{\tilde{M}}(X)$}\label{sec2}

In this section we introduce the four kind of sequential convergence on the space of Borel measures $\mathcal{\tilde{M}}(X)$. These notions differ from one another in strength, such that one can expect better properties on the sequence of measures or their limit measures under stronger convergent mode, at the cost of more difficulty in guaranteeing the stronger convergence. One can subtly choose the appropriate level of convergence mode, this is sometimes crucial in dealing with its problems. There are basically two points of views to describe these notions of convergence of sequences of measures. One is from their behaviours on measurable sets in $X$, another is from integrations of functions with respect to them. There are some equivalent descriptions of these notions from both point of views, however, the difficulty on verifying these descriptions may be of big difference in various applications, at least technically. This is one of the reason for us to provide more (equivalent, necessary or sufficient) conditions on these convergence. We decide to take the second point of view to define these notions of convergence, that is, we are going to define them by the asymptotic behaviours of integrations of functions of certain regularity with respect to the sequences of measures in $\mathcal{\tilde{M}}(X)$. 

For a set $A\subset X$, let $\overline{A}, A^o, A^c, \partial A$ be its closure, interior, complement and boundary respectively. A continuous function $f: X\rightarrow\mathbb{R}$ is said to \emph{vanish at infinity} if 
\begin{center}
$f^{-1}\big((-\infty,-a]\cup[a,\infty)\big)$ 
\end{center}
is always compact for any $a>0$ \cite[p132]{Fol} in $X$. The \emph{support} of a function $f: X\rightarrow\mathbb{R}$ is defined to be 
\begin{center}
$S(f):=\overline{f^{-1}\big((-\infty,0)\cup(0,\infty)\big)}$.  
\end{center}

The following families of functions are highlighted in our work. 

\begin{itemize}
\item $C(X)=\{f: f \mbox{ is a continuous function from } X \mbox{ to } \mathbb{R}\}.$
\item $C_0(X)=\{f: f \mbox{ is a continuous function from } X \mbox{ to } \mathbb{R} \mbox{ vanishing at infinity}\}$.
\item $C_c(X)=\{f: f \mbox{ is a continuous function from } X \mbox{ to } \mathbb{R} \mbox{ with compact support}\}$.
\item $C_b(X)=\{f: f \mbox{ is a bounded continuous function from } X \mbox{ to } \mathbb{R}\}$.
\item $M_b(X)=\{f: f \mbox{ is a bounded measurable function from } X \mbox{ to } \mathbb{R}\}$.
\item $M_\gamma(X)=\{f: f \mbox{ is a measurable function from } X \mbox{ to } [-\gamma,\gamma]\}$ for $\gamma\geq 0$.
\end{itemize}
The following relationships between these families of functions are obvious.
\begin{center}
$M_\gamma(X)\subset M_b(X)$ and  $C_c(X)\subset C_0(X)\subset C_b(X)\subset M_b(X)$.
\end{center}

For integrations on a topological space $X$ with a Borel measure, see \cite{Li} or \cite[Chapter 3]{Tay}.

\begin{definition}\label{def2}
For a sequence of measures $\{\nu_n\in\mathcal{\tilde{M}}(X)\}_{n=1}^\infty$, we say $\{\nu_n\}_{n=1}^\infty$ converges \emph{vaguely} to $\nu\in\mathcal{\tilde{M}}(X)$, denoted by $\nu_n\stackrel{v}{\rightarrow}\nu$, if
\begin{center}
$\lim_{n\rightarrow\infty }\int_X f(x) d\nu_n=\int_X f(x) d\nu$
\end{center}
for any  $f\in C_c(X)$.
\end{definition}

Denote by $\nu_n\stackrel{v}{\nrightarrow}\nu$ as $n\rightarrow\infty$ if $\{\nu_n\in\mathcal{\tilde{M}}(X)\}_{n=1}^\infty$ does not converge vaguely to $\nu\in\mathcal{\tilde{M}}(X)$.

\begin{remark}\label{rem1}
Note that some people prefer to define a sequence of measures $\{\nu_n\in\mathcal{\tilde{M}}(X)\}_{n=1}^\infty$ converges vaguely to  $\nu\in\mathcal{\tilde{M}}(X)$, if
\begin{center}
$\lim_{n\rightarrow\infty }\int_X f(x) d\nu_n=\int_X f(x) d\nu$
\end{center}
for any  $f\in C_0(X)$, see for example \cite[P223]{Fol} and \cite[P132]{Las1}. The two concepts are not equivalent to each other in general, see our Example \ref{exm1}. However, they two coincide with each other in some cases, see Proposition \ref{pro2}. Without special declarations we always mean Definition \ref{def2} by vague convergence in the following.
\end{remark}

A metric space is said to be a \emph{Heine-Borel} space if any closed bounded subset in it is compact \cite{JW}. There are lots of typical spaces in the Heine-Borel family, for example, the Euclidean spaces, or any $\sigma$-compact and locally compact metric space. For a measure $\mu\in\mathcal{\tilde{M}}(X)$, a set $A\subset X$ is called an $\mu$-continuity set if $\mu(\partial A)=0$. Following Kallenberg, we provide some equivalent conditions on verifying vague convergence of sequences of measures in $\mathcal{\hat{M}}(X)$ with $X$ being Heine-Borel.

\begin{theorem}\label{thm6}
For a sequence of measures $\{\nu_n\in\mathcal{\hat{M}}(X)\}_{n=1}^\infty$ and $\nu\in\mathcal{\hat{M}}(X)$ on a separable and complete Heine-Borel space $X$, the following conditions are equivalent.
\begin{enumerate}[(I).]
\item $\nu_n\stackrel{v}{\rightarrow}\nu$ as $n\rightarrow\infty$.
\item For any compact set $A\subset X$,
\begin{center}
$\limsup_{n\rightarrow\infty} \nu_n(A)\leq \nu(A)$, 
\end{center}
while for any bounded open set $B\subset X$,
\begin{center}
$\liminf_{n\rightarrow\infty} \nu_n(B)\geq \nu(B)$.
\end{center}
\item For any bounded closed set $A\subset X$,
\begin{center}
$\limsup_{n\rightarrow\infty} \nu_n(A)\leq \nu(A)$, 
\end{center}
while for any bounded open set $B\subset X$,
\begin{center}
$\liminf_{n\rightarrow\infty} \nu_n(B)\geq \nu(B)$.
\end{center}
\item For any bounded set $A\in \mathscr{B}$,
\begin{center}
$\nu_n(A^o)\leq \liminf_{n\rightarrow\infty} \nu_n(A)\leq \limsup_{n\rightarrow\infty} \nu_n(A)\leq \nu(\overline{A})$.
\end{center}
\item $\lim_{n\rightarrow\infty} \nu_n(A)=\nu(A)$ for any bounded $\nu$-continuity set $A\in \mathscr{B}$.
\item $\lim_{n\rightarrow\infty}\int_X fd\nu_n=\int_X fd\nu$ for any continuous function $f$ with bounded support.
\item $\lim_{n\rightarrow\infty}\int_X fd\nu_n=\int_X fd\nu$ for any H\"older continuous function $f\in C_c(X)$.
\item $\lim_{n\rightarrow\infty}\int_X fd\nu_n=\int_X fd\nu$ for any uniformly continuous function $f\in C_c(X)$.
\item $\lim_{n\rightarrow\infty}\int_X fd\nu_n=\int_X fd\nu$ for any function $f\in M_b(X)$ with bounded support and 
\begin{center}
$\nu(\{x\in X: f \mbox{ is discontinuous at } x\})=0$.
\end{center}
\item $\lim_{n\rightarrow\infty}\int_X fd\nu_n=\int_X fd\nu$ for any non-negative (or non-positive) valued function $f\in C_c(X)$.
\end{enumerate}
\end{theorem}

\begin{remark}
The conditions $(IV), (V)$ are due to Kallenberg \cite[Lemma 4.1]{Kallen1}, while $(IX)$ is due to Klenke \cite[Theorem 13.16]{Kle} essentially. Some of these conditions provide concrete examples of \emph{approximating classes} in $C_c(X)$. In fact, according to \cite[Lemma 4.1]{Kallen1}, condition $(V)$ can be substituted by  $\lim_{n\rightarrow\infty} \nu_n(A)=\nu(A)$ for any set in a \emph{dissecting semi-ring} of all bounded $\nu$-continuity sets.
\end{remark}

\begin{remark}
The requirement on $X$ being Heine-Borel may confine applications of the results. However, event for $X$ being not Heine-Borel, equivalence of some of the listed conditions still holds. For example, in case $X$ being merely a metric space, conditions $(I), (III), (IV)$ are still equivalent. Some conditions degenerates into only necessary or sufficient conditions with respect to $(I)$ without further assumptions on $X$, as one can judge from the proof of Theorem \label{thm6} oneself.   
\end{remark}

A stronger notion of sequential convergence of measures than vague convergence is the \emph{weak convergence} as following. 

\begin{definition}
For a sequence of measures $\{\nu_n\in\mathcal{\tilde{M}}(X)\}_{n=1}^\infty$, we say $\{\nu_n\}_{n=1}^\infty$ converges \emph{weakly} to $\nu\in\mathcal{\tilde{M}}(X)$, denoted by $\nu_n\stackrel{w}{\rightarrow}\nu$, if
\begin{center}
$\lim_{n\rightarrow\infty }\int_X f(x) d\nu_n=\int_X f(x) d\nu$
\end{center}
for any  $f\in C_b(X)$.
\end{definition}

There are lots of equivalent descriptions on the weak convergence of sequences of measures whose collection is called the \emph{Portemanteau Theorem}, see for example \cite[Theorem 2.1]{Bil1}, \cite[Theorem 1.4.16]{HL1} and \cite[Theorem 13.16]{Kle}. The weak convergence of sequences of measures is widely used and studied in various situations, so will not be our focus in this work. 

Although the weak convergence of sequences of measures is a powerful tool, there are some situations under which some important properties are not guaranteed under this mode of convergence, for example, the Vitali-Hahn-Saks Theorem (refer to \cite{Doo, HL2}) or the semi-continuity of some measure-dimension mappings (refer to \cite[Theorem 3.2]{Ma1}). These properties are guaranteed under the \emph{setwise convergence} of sequences of measures, which appears as a stronger mode of convergence  than the weak convergence as following.

\begin{definition}\label{def1}
For a sequence of measures $\{\nu_n\in\mathcal{\tilde{M}}(X)\}_{n=1}^\infty$, we say $\{\nu_n\}_{n=1}^\infty$ converges \emph{setwisely} to $\nu\in\mathcal{\tilde{M}}(X)$, denoted by $\nu_n\stackrel{s}{\rightarrow}\nu$, if
\begin{center}
$\lim_{n\rightarrow\infty }\int_X f(x) d\nu_n=\int_X f(x) d\nu$
\end{center}
for any $f\in M_b(X)$.
\end{definition}

Since simple functions are dense in the space of bounded measurable functions, $\nu_n\stackrel{s}{\rightarrow}\nu$ is equivalent to say that 
\begin{center}
$\lim_{n\rightarrow\infty }\nu_n(A)=\nu(A)$ 
\end{center}
for any $A\in\mathscr{B}$. Feinberg, Kasyanov and Zgurovsky gave some equivalent conditions on the setwise convergence of sequences of measures in $\mathcal{\hat{M}}(X)$ with a metric ambient space $X$ as following \cite[Theorem 2.3]{FKZ1}. 

\begin{theorem}[Feinberg-Kasyanov-Zgurovsky]\label{thm4}
For a sequence of measures $\{\nu_n\in\mathcal{\hat{M}}(X)\}_{n=1}^\infty$ and $\nu\in\mathcal{\hat{M}}(X)$ on a metric space $X$, the following conditions are equivalent to each other:
\begin{enumerate}[(I).]
\item $\nu_n\stackrel{s}{\rightarrow}\nu$ as $n\rightarrow\infty$.
\item $\lim_{n\rightarrow\infty} \nu_n(B)= \nu(B)$ for any open set $B\subset X$.
\item $\lim_{n\rightarrow\infty} \nu_n(A)= \nu(A)$ for any closed set $A\subset X$.
\end{enumerate}
\end{theorem}

\begin{remark}
Feinberg-Kasyanov-Zgurovsky's original result is set on sequences of probability measures in  $\mathcal{M}(X)$, however, their result extends naturally to sequences of finite measures in $\mathcal{\hat{M}}(X)$, or even sequences of infinite measures in $\mathcal{\tilde{M}}(X)$ in some cases. 
\end{remark}

Considering Theorem \ref{thm6} and \ref{thm3}, it is an interesting question to ask that when $\lim_{n\rightarrow\infty} \nu_n(A)= \nu(A)$ for any closed (or open) bounded set $A\subset X$ is enough to force $\nu_n\stackrel{s}{\rightarrow}\nu$ as $n\rightarrow\infty$ in $\mathcal{\hat{M}}(X)$, at least in case the ambient space $X$ is good enough. This is usually not true even if $X$ is a separable and complete Heine-Borel space $X$ (of course unbounded), see our Example \ref{exm2}.  

Since it is inevitable to deal with unbounded measurable functions in various applications of the setwise convergence, we will briefly discuss limit behaviours of integrations of unbounded measurable functions with respect to setwisely convergent sequences of measures in $\mathcal{\hat{M}}(X)$ in Section \ref{sec4}. Our Example \ref{exm3}. alerts the readers that the convergence of integrations of unbounded measurable functions with respect to setwisely convergent sequences of measures is usually lost. However, one may expect the convergence in some special cases, see Proposition \ref{pro3}.   

We are also quite interested in Theorem \ref{thm4} when the ambient space $X$ is non-metrizable. Since measures on non-metrizable ambient spaces lose regularity, one can expect that Theorem \ref{thm4} will not be true in some cases.  

\begin{theorem}\label{thm5}
For a topological space $X$ with its Borel $\sigma$-algebra $\mathscr{B}$, if it admits infinitely many non-empty pairwise disjoint closed (open) sets, and every proper closed (open) subset in $X$  contains at most finitely many disjoint non-empty closed (open) subsets, then there is a sequence of probability measures $\{\nu_n\in\mathcal{\hat{M}}(X)\}_{n=1}^\infty$ and  $\nu\in\mathcal{M}(X)$ satisfying both the conditions $(II), (III)$ in Theorem \ref{thm4}, while
\begin{center}
$\nu_n\stackrel{s}{\nrightarrow}\nu$ 
\end{center}
as $n\rightarrow\infty$.
\end{theorem}

Especially, this means $\lim_{n\rightarrow\infty} \nu_n(B)= \nu(B)$ for any open (or closed) set of an \emph{affine} (or \emph{projective}) \emph{space} endowed with the \emph{Zariski topology} is not enough to guarantee $\nu_n\stackrel{s}{\rightarrow}\nu$ for Borel measures $\{\nu_n\}_{n\in\mathbb{N}}\cup\{\nu\}$ on the affine (or projective) space, see Corollary \ref{cor1}. Theorem \ref{thm5} also provides some counter examples on denying Theorem \ref{thm6} and \ref{thm3} on non-metrizable ambient spaces. See also \cite[Example 1]{Pfa}.

The strongest notion of convergence of sequences of measures considered in our work is the \emph{TV convergence}. Of course more desiring properties is guaranteed under this mode of convergence.  

\begin{definition}
For a sequence of measures $\{\nu_n\in\mathcal{\tilde{M}}(X)\}_{n=1}^\infty$, we say $\{\nu_n\}_{n=1}^\infty$ converges \emph{in total variation} (TV) to $\nu\in\mathcal{\tilde{M}}(X)$, denoted by $\nu_n\stackrel{TV}{\rightarrow}\nu$, if
\begin{center}
$\lim_{n\rightarrow\infty }\sup_{f\in M_1(X)}|\int_X f(x) d\nu_n-\int_X f(x) d\nu|=0$.
\end{center}
\end{definition}

In this case the space $\mathcal{\tilde{M}}(X)$ is metrizable under the \emph{total-variation} (TV) metric
\begin{equation}\label{eq24}
\Vert\mu-\nu\Vert_{TV}:=2\sup_{A\in\mathscr{B}}|\mu(A)-\nu(A)|=\sup_{f\in M_1(X)}|\int_X f(x) d\mu-\int_X f(x) d\nu|
\end{equation}
for any two probability measures $\mu, \nu\in \mathcal{\tilde{M}}(X)$. We also give some equivalent conditions on the TV convergence of sequences of measures in Theorem \ref{thm3}.

\begin{theorem}\label{thm3}
For a sequence of measures $\{\nu_n\in\mathcal{\hat{M}}(X)\}_{n=1}^\infty$ and $\nu\in\mathcal{\hat{M}}(X)$ on a metric space $X$, the following conditions are equivalent to each other:
\begin{enumerate}[(I).]
\item $\nu_n\stackrel{TV}{\rightarrow}\nu$ as $n\rightarrow\infty$.
\item $\lim_{n\rightarrow\infty} \sup_{A \mbox{ is closed and bounded}}|\nu_n(A)-\nu(A)|=0$.
\item $\lim_{n\rightarrow\infty} \sup_{B \mbox{ is open and bounded}}|\nu_n(B)-\nu(B)|=0$.
\item $\lim_{n\rightarrow\infty}\sup_{f\in M_\gamma(X)}|\int_X f(x) d\nu_n-\int_X f(x) d\nu|=0$ for any $\gamma\geq 0$.
\item $\lim_{n\rightarrow\infty}\sup_{f \mbox{ has bounded support in } M_\gamma(X)}|\int_X f(x) d\nu_n-\int_X f(x) d\nu|=0$ for any $\gamma\geq 0$.
\item $\lim_{n\rightarrow\infty}\sup_{f \mbox{ has bounded support in } C(X)\cap M_\gamma(X)}|\int_X f(x) d\nu_n-\int_X f(x) d\nu|=0$ for any $\gamma\geq 0$.
\item $\lim_{n\rightarrow\infty}\sup_{f \mbox{ is uniformly continuous in } C(X)}|\int_X f(x) d\nu_n-\int_X f(x) d\nu|=0$ for any $\gamma\geq 0$.
\end{enumerate}
\end{theorem}

One is recommended to compare the result with \cite[Theorem 2.5]{FKZ1}, Theorem \ref{thm6} as well as our Example \ref{exm2}. 

We alert the readers that the strength of the modes of sequential convergence 
\begin{center}
TV convergence$\Rightarrow$ setwise convergence $\Rightarrow$ weak convergence $\Rightarrow$ vague convergence
\end{center}
holds as we are considering Borel measures on the topological space $X$. These relationships may not be true if one considers measures on $X$ with non-Borel $\sigma$-algebra. This also affects descriptions of these kinds of convergence of sequences of measures on $X$ with non-Borel $\sigma$-algebra.

\section{Description of the vague convergence of sequences of measures in $\mathcal{\hat{M}}(X)$}\label{sec3}

This section is devoted to the proof of Theorem \ref{thm6}. Before the proof we first make a comparison on the two kinds of vague convergence respectively in Definition \ref{def2} and Remark \ref{rem1}. Example \ref{exm1} shows that the one in Remark \ref{rem1} may be strictly stronger than the one in Definition \ref{def2} for vaguely sequential convergence in $\mathcal{\tilde{M}}(X)$. For a measure $\nu\in \mathcal{\tilde{M}}(X)$ and $A\in \mathscr{B}$,  let $\nu|_{A}$ be the restriction of $\nu$ on $A$.

\begin{example}\label{exm1}

Let $X=\{1,2,\dots\}=\mathbb{N}$ be endowed with the discrete topology and the corresponding Borel $\sigma$-algebra $\mathscr{B}$.  Let $\nu_c$ be the counting measure on $(X, \mathscr{B})$. Let $\nu_n:=\nu_c|_{\{n,n+1,\dots\}}$ for any $n\in\mathbb{N}$ on $(X, \mathscr{B})$. Let $\nu$ be the null measure on $(X, \mathscr{B})$.

Since a subset of $X$ is compact if and only if it is finite, for any function $f\in C_c(X)$, it is identically zero outside a finite subset. Therefore,
\begin{center}
$\int_X f d\nu_n=\sum_{m\ge n}f(m)\rightarrow 0=\int_X f d\nu$
\end{center}
as $n\rightarrow \infty.$

Now consider the continuous function $g(n)=\frac{1}{n}$ for any $n\in\mathbb{N}$, it is in $C_0(X)$, while
\begin{center}
$\int_X g d\nu_n=\sum_{m\ge n}\frac{1}{m}$
\end{center}
for any $n\in\mathbb{N}$, which does not converge to $0=\int_X gd\nu.$
\end{example}

Even if one restricts the consideration on the probability space $\mathcal{M}(X)$, the two versions of vague convergence may differ from each other if the ambient space $X$ is pathological, see Theorem \ref{thm5}. However,  the two notions are equivalent to each other on  $\mathcal{\hat{M}}(X)$ with a $\sigma$-compact and locally compact Hausdorff (LCH) space $X$.

\begin{proposition}\label{pro2}
In case $X$ is a $\sigma$-compact LCH space, let $\{\nu_n\}_{n\in\mathbb{N}}\cup\{\nu\}\subset \mathcal{\hat{M}}(X)$. If
\begin{center}
$\lim_{n\rightarrow\infty }\int_X f(x) d\nu_n=\int_X f(x) d\nu$
\end{center}
for any  $f\in C_c(X)$, then
\begin{equation}\label{eq22}
\lim_{n\rightarrow\infty }\int_X f(x) d\nu_n=\int_X f(x) d\nu
\end{equation}
for any  $f\in C_0(X)$.
\end{proposition}

\begin{proof}
If the ambient space $X$ is $\sigma$-compact and LCH, according to \cite[Section 1.10]{Tao}, for any  $f\in C_0(X)$, we can find $f_c\in C_c(X)$, such that
\begin{center}
$\Vert f-f_c\Vert_\infty\leq \epsilon$
\end{center}
for any $\epsilon>0$. Since $\lim_{n\rightarrow\infty }\int_X f_c d\nu_n=\int_X f_c d\nu$, there exists $N_1\in\mathbb{N}$ large enough, such that 
\begin{center}
$|\int_X f_c d\nu_n-\int_X f_c d\nu|<\epsilon$
\end{center}
for any $\epsilon>0$. So
\begin{center}
$
\begin{array}{ll}
& |\int_X f d\nu_n-\int_X f d\nu|\\
\leq & |\int_X f_c d\nu_n-\int_X f_c d\nu|+\int_X |f-f_c|(d\nu_n+d\nu)\\
\leq & (2+3\nu(X))\epsilon
\end{array}
$
\end{center}
for $n$ large enough. This implies (\ref{eq22}) since $\nu(X)<\infty$.
 
\end{proof}

From now on we go towards the proof of Theorem \ref{thm6}. To do this we need the following preceding results. For two subsets $A, B$ in a metric space $X$ with metric $\rho$, let 
\begin{center}
$\rho(A, B)=\inf_{x\in A, y\in B}\rho(x,y)$  
\end{center}
be their distance in $X$.

\begin{lemma}\label{lem4}
Let $\{\nu_n\}_{n\in\mathbb{N}}\cup\{\nu\}\subset \mathcal{\hat{M}}(X)$ with $X$ being Heine-Borel with metric $\rho$. If $\nu_n\stackrel{v}{\rightarrow}\nu$ as $n\rightarrow\infty$, then
\begin{equation}\label{eq8}
\limsup_{n\rightarrow\infty} \nu_n(A)\leq \nu(A)
\end{equation}
for any bounded closed set $A\subset X$, while 
\begin{equation}\label{eq14}
\liminf_{n\rightarrow\infty} \nu_n(B)\geq \nu(B)
\end{equation}
for any bounded open set $B\subset X$.
\end{lemma}
\begin{proof}
Let $A\subset X$ be an arbitrarily bounded closed set. For any $n\in\mathbb{N}$, consider the following function,
\begin{equation}\label{eq12}
f_{n, A}=\left\{\begin{array}{ll}
1 & x\in A, \vspace{2mm}\\
1-n\rho(x, A) &  0<\rho(x, A)<\frac{1}{n}, \vspace{2mm}\\
0 &  \rho(x, A)\geq\frac{1}{n}.
\end{array}\right.
\end{equation}

Note that $f_{n, A}$ has bounded support, and so compact support since $X$ is Heine-Borel for any $n\in\mathbb{N}$. Then 
\begin{center}
$\limsup_{n\rightarrow\infty} \nu_n(A)\leq \limsup_{n\rightarrow\infty} \int_X f_{n, A} d\nu_n=\int_X f_{n, A} d\nu$
\end{center}
as $\nu_n\stackrel{v}{\rightarrow}\nu$. Since the sequence $\{\int_X f_{n, A} d\nu\}_{n=1}^\infty$ decreases to $\nu(A)$ as $n\rightarrow\infty$, we get (\ref{eq8}).

Now let $B\subset X$ be an arbitrary bounded open set. For any $n\in\mathbb{N}$, consider the following function,
\begin{equation}\label{eq23}
g_{n, B}=\left\{\begin{array}{ll}
1 & x\in B \mbox{ and } \rho(x, \partial B)\geq \frac{1}{n}, \vspace{2mm}\\
n\rho(x, \partial B) &  x\in B \mbox{ and } \rho(x, \partial B)<\frac{1}{n}, \vspace{2mm}\\
0 &  x\notin B
\end{array}\right.
\end{equation}
for $n\in\mathbb{N}$ large enough. $g_{n, B}$ has bounded support and so compact support since $X$ is Heine-Borel for any $n\in\mathbb{N}$. Then 
\begin{center}
$\liminf_{n\rightarrow\infty} \nu_n(B)\geq \liminf_{n\rightarrow\infty} \int_X g_{n, B} d\nu_n=\int_X g_{n, B} d\nu$
\end{center}
as $\nu_n\stackrel{v}{\rightarrow}\nu$. Since the sequence $\{\int_X g_{n, B} d\nu\}_{n=1}^\infty$ increases to $\nu(B)$ as $n\rightarrow\infty$, we get (\ref{eq14}).  

\end{proof}

\begin{remark}\label{rem2}
Note that if we assume $X$ is a bounded Heine-Borel space in Lemma \ref{lem4}, then (\ref{eq8}) holds on any bounded closed set $A\subset X$ is equivalent to that (\ref{eq14}) holds on any bounded open set $B\subset X$.
\end{remark}

The following result interprets some limit behaviours of sequences of measures on bounded continuity sets in words of limit behaviours of sequences of those measures on bounded measurable sets in a metric space.
\begin{lemma}[Kallenberg]\label{lem3}
Let $\{\nu_n\}_{n\in\mathbb{N}}\cup\{\nu\}\subset \mathcal{\hat{M}}(X)$ for a metric space $(X, \rho)$. Then
\begin{equation}\label{eq9}
\nu_n(A^o)\leq \liminf_{n\rightarrow\infty} \nu_n(A)\leq \limsup_{n\rightarrow\infty} \nu_n(A)\leq \nu(\overline{A})
\end{equation}
for any bounded set $A\in \mathscr{B}$ if and only if
\begin{equation}\label{eq10}
\lim_{n\rightarrow\infty} \nu_n(A)=\nu(A)
\end{equation}
for any bounded $\nu$-continuity set $A\in \mathscr{B}$.
\end{lemma}

\begin{proof}
If (\ref{eq9}) holds on any bounded set $A\in \mathscr{B}$, then (\ref{eq10}) holds on any bounded $\nu$-continuity sets obviously. In the following we show the converse. Suppose (\ref{eq10}) holds for any bounded $\nu$-continuity set $A\in \mathscr{B}$. We first justify the first inequality in (\ref{eq9}). For any bounded measurable set $A$, since $\partial A^o=\emptyset$, $A^o$ is a $\nu$-continuity set, this induces
\begin{equation}
\liminf_{n\rightarrow\infty} \nu_n(A) \geq\lim_{n\rightarrow\infty} \nu_n(A^o)=\nu(A^o).
\end{equation}  

Now we justify the third inequality in (\ref{eq9}), which is enough to finish the proof. We do this by reduction to absurdity. Suppose $\limsup_{n\rightarrow\infty} \nu_n(A)> \nu(\overline{A})$ for some bounded set $A\in \mathscr{B}$. Consider the following collection of bounded open sets in $X$,  
\begin{center}
$\big\{A_\epsilon:=\{x\in X: \rho(x, A)<\epsilon\}\big\}_{0<\epsilon<\infty}$.
\end{center}
It is obvious that 
\begin{center}
$\limsup_{\epsilon\rightarrow 0} A_\epsilon=\overline{A}$,
\end{center}
which forces 
\begin{equation}\label{eq11}
\lim_{\epsilon\rightarrow 0} \nu(A_\epsilon)=\nu(\overline{A})
\end{equation}
since $\nu$ is finite. Now choose a decreasing sequence of positive real numbers $\{\epsilon_k\}_{k=1}^\infty$ with $\lim_{k\rightarrow\infty} \epsilon_k=0$, such that $\{A_{\epsilon_k}\}_{k=1}^\infty$ are all $\nu$-continuity set for any $k\in\mathbb{N}$ (in fact there can be at most countably many $\epsilon$ in $(0,\infty)$ such that $A_\epsilon$ is not an $\nu$-continuity set). Then we have
\begin{center}
$\nu(A_{\epsilon_k})=\lim_{n\rightarrow\infty} \nu_n(A_{\epsilon_k})\geq \limsup_{n\rightarrow\infty} \nu_n(A)> \nu(\overline{A})$
\end{center} 
for any $k\in\mathbb{N}$. This gives
\begin{equation}
\lim_{k\rightarrow\infty} \nu(A_{\epsilon_k})=\lim_{\epsilon_k\rightarrow 0} \nu(A_{\epsilon_k})> \nu(\overline{A}),
\end{equation}
which contradicts (\ref{eq11}).
\end{proof}

One is recommended to refer to \cite[Lemma 4.1]{Kallen1} for condition (\ref{eq9}). In fact (\ref{eq9}) holds on any set $A\in \mathscr{B}$ is equivalent to that (\ref{eq10}) holds on $\nu$-continuity sets in the context of Lemma \ref{lem3}. 

\begin{lemma}\label{lem5}
Let $\{\nu_n\}_{n\in\mathbb{N}}\cup\{\nu\}\subset \mathcal{\hat{M}}(X)$ with a metric space $(X, \rho)$. If 
\begin{center}
$\lim_{n\rightarrow\infty}\int_X fd\nu_n=\int_X fd\nu$ 
\end{center}
for any H\"older continuous function $f\in C_c(X)$, then 
\begin{equation}
\limsup_{n\rightarrow\infty} \nu_n(A)\leq \nu(A)
\end{equation}
for any bounded closed set $A\subset X$, while 
\begin{equation}
\liminf_{n\rightarrow\infty} \nu_n(B)\geq \nu(B)
\end{equation}
for any bounded open set $B\subset X$.
\end{lemma} 
\begin{proof}
This is because the continuous functions $f_{n, A}, g_{n, B} \in C_c(X)$ in (\ref{eq12}) and (\ref{eq23}) are both H\"older continuous for any $n\in\mathbb{N}$ large enough.  
\end{proof}

Equipped with all the above results, now we are well prepared to prove our Theorem \ref{thm6}.\\

Proof of Theorem \ref{thm6}:\\

\begin{proof}
The strategy of our proof follows the following diagram.

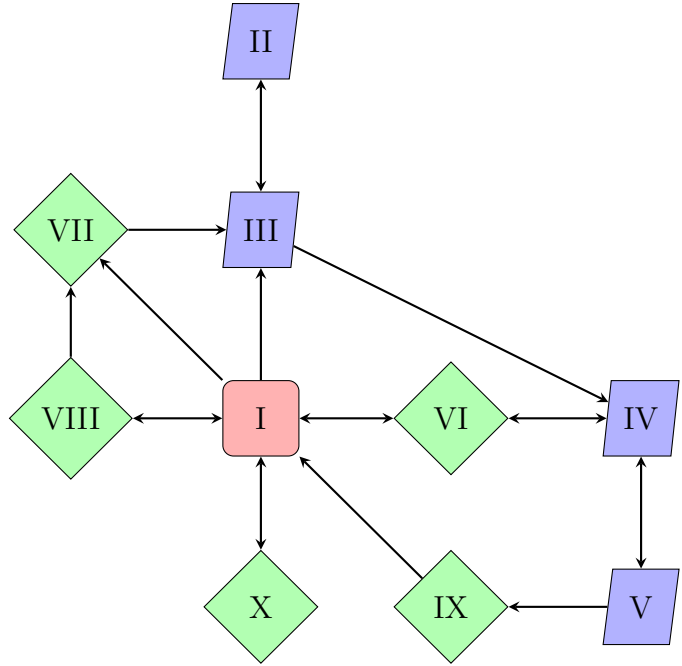
\tikzstyle{startstop} = [rectangle, rounded corners, minimum width=1cm, minimum height=1cm,text centered, draw=black, fill=red!30]

\tikzstyle{io} = [trapezium, trapezium left angle=80, trapezium right angle=100, minimum width=1cm, minimum height=1cm, text centered, draw=black, fill=blue!30, trapezium stretches=true]

\tikzstyle{process} = [rectangle, minimum width=1cm, minimum height=1cm, text centered, draw=black, fill=orange!30]

\tikzstyle{decision} = [diamond, minimum width=1.5cm, minimum height=1.5cm, text centered, draw=black, fill=green!30]

\tikzstyle{arrow1} = [thick,->,>=stealth]
\tikzstyle{arrow2} = [thick, >=triangle 45, <->,>=stealth]

\begin{center}
\begin{tikzpicture}[node distance=2.5cm]
\node(I)[startstop]{I};
\node (VI) [decision, right of=I] {VI};
\node (IV) [io, right of=VI] {IV};
\node (V) [io, below of=IV] {V};
\node (IX) [decision, left of=V] {IX};
\node (X) [decision, below of=I] {X};
\node (VIII) [decision, left of=I] {VIII};
\node (VII) [decision, above of=VIII] {VII};
\node (III) [io, right of=VII] {III};
\node (II) [io, above of=III] {II};

\draw [arrow2] (I) -- (VI);
\draw [arrow2] (VI) -- (IV);
\draw [arrow2] (IV) -- (V);
\draw [arrow1] (V) -- (IX);
\draw [arrow1] (IX) -- (I);
\draw [arrow2] (I) -- (X);
\draw [arrow1] (VIII) -- (VII);
\draw [arrow1] (VII) -- (III);
\draw [arrow2] (I) -- (VIII);
\draw [arrow1] (I) -- (III);
\draw [arrow1] (I) -- (VII);
\draw [arrow1] (III) -- (IV);
\draw [arrow2] (II) -- (III);
\end{tikzpicture}
\end{center}

\begin{itemize}

\item $(I)\Leftrightarrow(VI)$: this is because $X$ is Heine-Borel.

\item $(IV)\Leftrightarrow(VI)$: this is due to \cite[Lemma 4.1]{Kallen1}.

\item $(IV)\Leftrightarrow(V)$: this is due to Lemma \ref{lem3}.

\item $(V)\Rightarrow(IX)$: see \cite[Theorem 13.16 $(vi)\Rightarrow(iii)$]{Kle}. The proof applies to our bounded case here.

\item $(IX)\Rightarrow(I)$: this is trivial.

\item $(I)\Rightarrow(VII)$: this is trivial.

\item $(VIII)\Rightarrow(VII)$: this is because any H\"older continuous function on a metric space is uniformly continuous.

\item $(I)\Leftrightarrow(VIII)$: this is because any continuous function on a compact metric space is uniformly continuous.

\item $(VII)\Rightarrow(III)$: this is due to Lemma \ref{lem5}.

\item $(I)\Rightarrow(III)$: this is due to Lemma \ref{lem4}.

\item $(III)\Rightarrow(IV)$: this is obvious.

\item $(II)\Leftrightarrow(III)$: this is because $X$ is Heine-Borel.

\item $(I)\Leftrightarrow(X)$: this is because any $f\in C_c(X)$ can be splitted as the difference of two 
non-negative (or non-positive) valued function in $C_c(X)$.
\end{itemize}

\end{proof}

Due to Remark \ref{rem2}, in case $X$ is a bounded Heine-Borel space in Theorem \ref{thm6}, condition $(III)$ degenerates into that either (\ref{eq8}) holds on any bounded closed set $A\subset X$ or  (\ref{eq14}) holds on any bounded open set $B\subset X$.

\section{Description of the setwise convergence of sequences of measures in $\mathcal{\hat{M}}(X)$}\label{sec4}

In this section we focus on the equivalent descriptions of setwise convergence of sequences of bounded Borel measures with the ambient space $X$ being a general topological space. We first give an example to show that 
\begin{center}
$\lim_{n\rightarrow\infty} \nu_n(A)= \nu(A)$ 
\end{center}
for any closed (or open) bounded set $A\subset X$ is not enough to guarantee its setwise convergence in $\mathcal{\tilde{M}}(X)$. Then we give an example that the sequence of integrations of an unbounded measurable function with respect to a setwisely convergent sequence of measures in $\mathcal{\hat{M}}(X)$ diverge, followed by a partially convergent result. At last we prove Theorem \ref{thm5} and indicate some of its applications. Let $\mathfrak{L}_d$ be the $d$-dimensional Lebesgue measure on $\mathbb{R}^d$.

\begin{example}\label{exm2}
Let $X=[1,\infty)$ endowed with the Euclidean metric. Consider the following sequence of Borel measures
on $(X,\mathscr{B})$,
\begin{center}
$\nu_n(A)=\int_{[1,n]\cap A}\frac{1}{x^4}dx+\mathfrak{L}_1|_{[n,n+1]}(A)$
\end{center}
for any $A\in\mathscr{B}$, in which $\mathfrak{L}_1|_{[n,n+1]}$ is the restriction of  $\mathfrak{L}_1$ on $[n,n+1]$. Let
\begin{center}
$\nu(A)=\int_{X\cap A} \frac{1}{x^4}dx$
\end{center}
for any $A\in\mathscr{B}$.
\end{example}

In Example \ref{exm2} $X$ is a separable and complete Heine-Borel space. One can check easily that 
\begin{center}
$\lim_{n\rightarrow\infty} \nu_n(A)= \nu(A)$ 
\end{center}
for any closed (or open) bounded set $A\in\mathscr{B}$ (in fact $\nu_n(A)= \nu(A)$ for any $n$ large enough). However, $\{\nu_n\}_{n\in\mathbb{N}}$ does not converge setwisely to $\nu$ since 
\begin{center}
$\lim_{n\rightarrow\infty} \nu_n(X)= \nu(X)+1$. 
\end{center}

\begin{example}\label{exm3}
Let $X=[1,\infty)$ endowed with the Euclidean metric. Consider the following sequence of Borel measures
on $(X,\mathscr{B})$,
\begin{center}
$
\nu_n(A)=\left\{\begin{array}{ll}
\int_{[1,n]\cap A}\frac{1}{x^4}dx+\frac{1}{n^2}\mathfrak{L}_1|_{[n,n+1]}(A) & n \mbox{ is odd}, \vspace{2mm}\\
\int_{[1,n]\cap A}\frac{1}{x^4}dx+\frac{2}{n^2}\mathfrak{L}_1|_{[n,n+1]}(A) & n \mbox{ is even }
\end{array}\right.
$
\end{center}
for any $A\in\mathscr{B}$. Let
\begin{center}
$\nu(A)=\int_{X\cap A} \frac{1}{x^4}dx$
\end{center}
for any $A\in\mathscr{B}$.
\end{example}

One can check easily that in Example \ref{exm3} that $\nu_n\stackrel{s}{\rightarrow}\nu$ as $n\rightarrow\infty$ (in fact the convergence is even TV). Let $f(x)=x^2$ on $X$. Now consider the  integrations of the function $f$ with respect to $\{\nu_n\}_{n\in\mathbb{N}}$ and $\nu$.  One can check that 
\begin{center}
$
\lim_{n\rightarrow\infty} \int_X fd\nu_n=\left\{\begin{array}{ll}
\frac{4}{3} & n \mbox{ is odd}, \vspace{2mm}\\
\frac{7}{3} & n \mbox{ is even},
\end{array}\right.
$
\end{center}
while $\int_X fd\nu=\frac{1}{3}$. 

\begin{remark}
The measures in Example \ref{exm2} and \ref{exm3} can be made into probability ones with continuous or even smooth density with respect to $\mathfrak{L}_1$ on $[1,\infty)$. 
\end{remark}

This example indicates the complication of integral behaviours on families of unbounded measurable test functions with respect to setwisely convergent sequences of finite measures. However, in some special cases one can expect the convergence of integrals of test functions with respect to setwisely convergent sequences of measures. For a function $f: X\rightarrow \mathbb{R}$ on a topological space $X$, let
\begin{center}
$
f^+(x)=\left\{\begin{array}{ll}
f(x) & \mbox{ for } f(x)\geq 0,\\
0 & \mbox{ for } f(x)< 0
\end{array}\right.$
\end{center}
and
\begin{center}
$
f^-(x)=\left\{\begin{array}{ll}
0 & \mbox{ for } f(x)> 0,\\
-f(x) &  \mbox{ for } f(x)\leq 0
\end{array}\right.$
\end{center}
be its \emph{positive part} and \emph{negative part} respectively. So $f(x)=f^+(x)-f^-(x)$ on $X$. We say $f$ is \emph{integrable} with respect to a Borel measure $\nu$ on $X$ if at least one of the integrals $\int_X f^+d\nu$ and $\int_X f^-d\nu$ is finite.    

\begin{proposition}\label{pro3}
Let $X$ be a topological space and $\nu_n\stackrel{s}{\rightarrow}\nu$ as $n\rightarrow\infty$ in $\mathcal{\hat{M}}(X)$. For an unbounded function $f: X\rightarrow [0,\infty)$, if $\int_X fd\nu=\infty$, then
\begin{equation}\label{eq16}
\lim_{n\rightarrow\infty} \int_X fd\nu_n=\infty.
\end{equation}
\end{proposition}

\begin{proof}
Let $A_k=\{x\in X: 0\leq f(x)<k\}$. Since $\int_X fd\nu=\lim_{k\rightarrow\infty} \int_{A_k} fd\nu=\infty$, for any $M>0$, there exists some $k_M\in\mathbb{N}$ large enough, such that   
\begin{center}
$\int_{A_{k_M}} fd\nu> M$.
\end{center}
Now consider the truncated bounded measurable function
\begin{center}
$
1_{A_{k_M}}f(x)=\left\{\begin{array}{ll}
f(x) & \mbox{ for } f(x)< k_M,\\
0 &  \mbox{ for } f(x)\geq k_M
\end{array}\right.$
\end{center}
on $X$, in which $1_{A_{k_M}}$ is the characteristic function of $A_{k_M}$. Since $\nu_n\stackrel{s}{\rightarrow}\nu$ as $n\rightarrow\infty$, we have
\begin{center}
$\lim_{n\rightarrow\infty} \int_X fd\nu_n\geq \lim_{n\rightarrow\infty} \int_X 1_{A_{k_M}}fd\nu_n=\int_X 1_{A_{k_M}}fd\nu\geq\int_{A_{k_M}} fd\nu> M$.
\end{center} 
This justifies (\ref{eq16}).
\end{proof}

\begin{remark}
Proposition \ref{pro3} does not apply to unbounded integrable functions $f: X\rightarrow [0,\infty)$ whose integration is finite with respect to $\nu$, as one can construct a counter example based on Example \ref{exm3}. This again indicates the complication of integral behaviours on families of unbounded measurable test functions with respect to setwisely convergent sequences of measures. 
\end{remark}

The rest of the section is devoted to descriptions of setwise convergence of sequences of measures in $\mathcal{\hat{M}}(X)$ with non-metrizable ambient space $X$. Let $\delta_x$ be the Dirac measure at the point $x\subset X$. We first give a proof of Theorem \ref{thm5}. \\

Proof of Theorem \ref{thm5}:\\

\begin{proof}
We only prove the case on existence of closed sets with the desiring properties, the open case is similar to the closed case.

According to the assumption on the topology of $X$, let $\{A_i\}_{i=1}^\infty$ be an infinite sequence of non-empty and pairwise disjoint closed sets in $X$ such that
\begin{center}
$X\setminus A\neq\emptyset$,
\end{center}
in which $A=\cup_{i=1}^\infty\{A_i\}$.
Now choose a sequence of points $\{a_i\in A_i\}_{i=1}^\infty$ and $a_*\in X\setminus A$. Define a sequence of probability measures $\{\nu_n\in\mathcal{M}(X)\}_{n=1}^\infty$ as following on $X$,
\begin{center}
$\nu_n=\cfrac{n-1}{n}\delta_{a_n}+\cfrac{1}{n}\delta_{a_*}$.
\end{center}
Let $\nu=\delta_{a_*}$. We claim that the sequence of measures $\{\nu_n\}_{n=1}^\infty$ and $\nu$ satisfy both the conditions $(II), (III)$ in Theorem \ref{thm4}. To see this, for any proper closed subset $F\subset X$, it intersects with at most finitely many sets from $\{A_n\}_{n\in\mathbb{N}}$. If this is not true, we can always find a sub-sequence $\{n_j\}_{j\in\mathbb{N}}$, such that $\{F\cap A_{n_j}\}_{j\in\mathbb{N}}$ are all closed subset of $F$, which contradicts the assumption that every proper closed set contains at most finitely many disjoint non-empty closed subsets. Now we can see that 
\begin{center}
$\lim_{n\rightarrow\infty}\nu_n(F)=0$.
\end{center}
This forces
\begin{center}
$\lim_{n\rightarrow\infty}\nu_n(G)=1$.
\end{center}
for any non-empty open set $G\subset X$, which justifies the claim. However, it is easy to see that
\begin{center}
$\nu_n\stackrel{s}{\nrightarrow}\nu$
\end{center}
as $n\rightarrow\infty$ since
\begin{center}
$\lim_{n\rightarrow\infty}\nu_n(A)=1>\nu(A)=0$.
\end{center}

\end{proof}

Considering Theorem \ref{thm4}, we can deduce the following result in virtue of Theorem \ref{thm5}. One can clearly see the impact of the topology of the ambient space $X$ on the sequential convergence of measures on $X$ from the result.

\begin{corollary}
If a topological space $X$ admits infinitely many pairwise disjoint closed (open) sets, and every proper closed (open) subset in $X$ contains at most finitely many disjoint non-empty closed (open) subsets, then it is not metrizable.
\end{corollary}

Theorem \ref{thm5} has some applications to some well-known (non-metrizable) topological spaces in various circumstances. For example, considering the Zariski topology on algebraic varieties (see for example \cite{Har}, the topology is known to be non-metrizable), we have the following result.

\begin{corollary}\label{cor1}
For $K$ being an algebraically closed field, let $X=\mathbb{A}^n$ (or $\mathbb{P}^n$) be the $n$-dimensional affine (or projective) space over $K$ for some $n\geq 2$. Let $S_n$ (or $S_{n+1}$) be a set of polynomials of $n$ (or $n+1$) variables over $K$ sharing infinitely many common solutions for some $n\geq 2$. Then the affine (projective) space equipped with the Zariski topology represented by the triples
\begin{center}
$(\mathbb{A}^n, S_n, K)$ (or $\big(\mathbb{P}^n, S_{n+1}, K)\big)$
\end{center}
admits a sequence of probability measures $\{\nu_n\in\mathcal{M}(X)\}_{n=1}^\infty$ and  $\nu\in\mathcal{M}(X)$ satisfying both the conditions $(II), (III)$ in Theorem \ref{thm4}, while
\begin{center}
$\nu_n\stackrel{s}{\nrightarrow}\nu$
\end{center}
as $n\rightarrow\infty$.
\end{corollary}

\begin{proof}
This is because $\mathbb{A}^n$ (or $\mathbb{P}^n$) with the Zariski topology admits infinitely many disjoint closed sets as common solutions of the polynomials in $S_n$ (or $S_{n+1}$) for any $n\geq 2$. Moreover, every closed set in $\mathbb{A}^n$ (or $\mathbb{P}^n$) is constituted by only finitely many solutions of the polynomials in $S_n$ (or $S_{n+1}$). Then the conclusion follows from  Theorem  \ref{thm5} instantly. 
\end{proof}

It is easy to construct a sequence of  probability measures $\{\nu_n\in\mathcal{M}(X)\}_{n=1}^\infty$ and  $\nu$ in $\mathcal{M}(X)$ on $(X=\mathbb{A}^n, S_n, K)$ (or $(X=\mathbb{P}^n, S_{n+1}, K)$) with the Zariski topology, such that it satisfies all the conditions $(ii)-(v)$ in \cite[Theorem 2.1]{Bil1}, while
\begin{center}
$\nu_n\stackrel{w}{\nrightarrow}\nu$,
\end{center}
following the ideas of Proof of Theorem \ref{thm4}. The details are left to the readers. So we have the following result.

\begin{proposition}\label{pro1}
For $K$ being an algebraically closed field, let $X=\mathbb{A}^n$ (or $\mathbb{P}^n$) be the $n$-dimensional affine (or projective) space over $K$ for some $n\geq 2$. Let $S_n$ (or $S_{n+1}$) be a set of polynomials of $n$ (or $n+1$) variables over $K$ sharing infinitely many common solutions for some $n\geq 2$. Then the affine (projective) space equipped with the Zariski topology represented by the triples
\begin{center}
$(\mathbb{A}^n, S_n, K)$ (or $\big(\mathbb{P}^n, S_{n+1}, K)\big)$
\end{center}
admits a sequence of probability measures $\{\nu_n\in\mathcal{M}(X)\}_{n=1}^\infty$ and  $\nu\in\mathcal{M}(X)$ satisfying all the conditions $(ii)-(v)$ in \cite[Theorem 2.1]{Bil1}, while
\begin{center}
$\nu_n\stackrel{w}{\nrightarrow}\nu$
\end{center}
as $n\rightarrow\infty$.
\end{proposition}

This result denies the Portemanteau Theorem on non-metrizable ambient spaces, in its general form.

\section{Description of the TV convergence of sequences of measures in $\mathcal{\hat{M}}(X)$}\label{sec6}

In this section we prove Theorem \ref{thm3}, followed by a similar discussion on the descriptions of TV sequential convergence of measures in $\mathcal{\hat{M}}(X)$ with the ambient space being non-metrizable. To prove Theorem \ref{thm3}, we need several preliminary results on the equivalent description of TV distance between two measures in $\mathcal{\hat{M}}(X)$.
\begin{lemma}\label{lem1}
For two finite measures $\mu, \nu\in\mathcal{\hat{M}}(X)$ with the ambient space $X$ being metrizable, we have
\begin{center}
$
\begin{array}{ll}
& \sup_{A\in\mathscr{B}}|\mu(A)-\nu(A)|\\
=&\sup_{A\mbox{ is closed and bounded}}|\mu(A)-\nu(A)|\\
=&\sup_{B\mbox{ is open and bounded}}|\mu(B)-\nu(B)|.
\end{array}
$
\end{center}
\end{lemma}
\begin{proof}
According to \cite[Theorem 2.5 (i)(ii)]{FKZ1}, to justify these equalities, we only need to show
\begin{equation}\label{eq1}
\sup_{A\mbox{ is closed}}|\mu(A)-\nu(A)|=\sup_{A\mbox{ is closed and bounded}}|\mu(A)-\nu(A)|
\end{equation}
and
\begin{equation}\label{eq2}
\sup_{B \mbox{ is open}}|\mu(B)-\nu(B)|=\sup_{B\mbox{ is open and bounded}}|\mu(B)-\nu(B)|.
\end{equation}
In the following we only show (\ref{eq1}), as the proof of (\ref{eq2}) follows a similar way. Choose an arbitrary point $x_0\in X$, consider the sequence of closed balls $\{\overline{B(x_0,n)}\}_{n\in\mathbb{N}}$ centred at $x_0$  of radius $n\in\mathbb{N}$. Since $\cup_{n\in\mathbb{N}}\overline{B(x_0,n)}=X$, for any small $\epsilon>0$, there exists $n_\epsilon\in\mathbb{N}$ large enough, such that
\begin{equation}\label{eq3}
\mu(\overline{B(x_0,n)})>\mu(X)-\epsilon \mbox{ and } \nu(\overline{B(x_0,n)})>\nu(X)-\epsilon.
\end{equation}
Now let
\begin{center}
 $\sup_{A\mbox{ is closed}}|\mu(A)-\nu(A)|:=s$.
\end{center}
Then for any small $\epsilon>0$, there exists a closed set $A\subset X$ such that
\begin{center}
$0\leq s-|\mu(A)-\nu(A)|<\epsilon$.
\end{center}
Let $A_\epsilon=A\cap \overline{B(x_0,n)}$. It is a bounded and closed set. Since
\begin{center}
$\mu(A)-\nu(A)=\mu(A_\epsilon)-\nu(A_\epsilon)+\mu(A\setminus A_\epsilon)-\nu(A\setminus A_\epsilon)$,
\end{center}
then
\begin{center}
$
\begin{array}{ll}
& |\mu(A)-\nu(A)|-|\mu(A_\epsilon)-\nu(A_\epsilon)| \\
\leq & |\mu(A\setminus A_\epsilon)-\nu(A\setminus A_\epsilon)|\\
\leq & \max\{\mu(A\setminus A_\epsilon), \nu(A\setminus A_\epsilon)\}\\
\leq & \max\{\mu(X\setminus \overline{B(x_0,n)}), \nu(X\setminus \overline{B(x_0,n)})\}\\
\leq & \epsilon.
\end{array}
$
\end{center}
The last inequality is due to (\ref{eq3}). Then we have
\begin{center}
$
0\leq s-|\mu(A_\epsilon)-\nu(A_\epsilon)| =s-|\mu(A)-\nu(A)|+|\mu(A)-\nu(A)|-|\mu(A_\epsilon)-\nu(A_\epsilon)|<2\epsilon,
$
\end{center}
which is enough to imply (\ref{eq1}).
\end{proof}

\begin{remark}\label{rem3}
In case $X$ is a $\sigma$-compact metric space, there is an increasing sequence of compact subsets $\{K_n\}_{n\in\mathbb{N}}$ whose union is $X$. Then
\begin{equation}
\sup_{A\mbox{ is compact}}|\mu(A)-\nu(A)|=\sup_{A\mbox{ is closed}}|\mu(A)-\nu(A)|.
\end{equation}
This can be seen by repeating the proof of Lemma \ref{lem1} with $\overline{B(x_0,n)}$ replaced by $K_n$ therein. $A_\epsilon$ is now a compact set since it is a closed subset of a compact set.
\end{remark}

\begin{lemma}\label{lem2}
For any non-negative $\gamma\in\mathbb{R}$, we have
\begin{center}
$\sup_{f\in M_\gamma(X)}|\int_X f(x) d\mu-\int_X f(x) d\nu|=\gamma\sup_{f\in M_1(X)}|\int_X f(x) d\mu-\int_X f(x) d\nu|$
\end{center}
for any $\mu, \nu\in\mathcal{\hat{M}}(X)$ with $X$ being a topological space.
\end{lemma}
\begin{proof}
Without loss of generality we assume $\gamma>0$. We only show
\begin{equation}\label{eq6}
\sup_{f\in M_\gamma(X)}\Big|\int_X f(x) d\mu-\int_X f(x) d\nu\Big|\leq \gamma\sup_{f\in M_1(X)}\Big|\int_X f(x) d\mu-\int_X f(x) d\nu\Big|
\end{equation}
in the following, the inverse inequality follows a similar way. For any small $\epsilon>0$, there exists $f_\epsilon\in M_\gamma(X)$ such that
\begin{equation}\label{eq4}
0\leq \sup_{f\in M_\gamma(X)}\Big|\int_X f(x) d\mu-\int_X f(x) d\nu\Big|-\Big|\int_X f_\epsilon(x) d\mu-\int_X f_\epsilon(x) d\nu\Big|<\epsilon.
\end{equation}
Since $\cfrac{1}{\gamma}f_\epsilon\in M_1(X)$, we have
\begin{equation}\label{eq5}
\gamma\sup_{f\in M_1(X)}\Big|\int_X f(x) d\mu-\int_X f(x) d\nu\Big|\geq \Big|\int_X f_\epsilon(x) d\mu-\int_X f_\epsilon(x) d\nu\Big|.
\end{equation}
Now combing (\ref{eq4}) and (\ref{eq5}) together, we have
\begin{equation}\label{eq7}
\begin{array}{ll}
& \sup_{f\in M_\gamma(X)}|\int_X f(x) d\mu-\int_X f(x)d\nu|-\gamma\sup_{f\in M_1(X)}|\int_X f(x) d\mu-\int_X f(x) d\nu| \vspace{2mm}\\
\leq & \sup_{f\in M_\gamma(X)}|\int_X f(x) d\mu-\int_X f(x) d\nu|-|\int_X f_\epsilon(x) d\mu-\int_X f_\epsilon(x) d\nu| \vspace{2mm}\\
< &\epsilon.
\end{array}
\end{equation}
Then we get (\ref{eq6}) by letting $\epsilon\rightarrow 0$ in (\ref{eq7}).
\end{proof}

\begin{lemma}\label{lem6}
Let $X$ be a metric space endowed with the metric $\rho$. For any non-negative $\gamma\in\mathbb{R}$, we have
\begin{center}
$
\begin{array}{ll}
& \sup_{f\in M_\gamma(X)}|\int_X f(x) d\mu- \int_X f(x)d\nu|\vspace{2mm}\\
=&\sup_{f \mbox{ has bounded support in } C(X)\cap M_\gamma(X)}|\int_X f(x) d\mu-\int_X f(x) d\nu|
\end{array}
$
\end{center}
for any $\mu, \nu\in\mathcal{\hat{M}}(X)$.
\end{lemma}
\begin{proof}
Considering \cite[Theorem 2.5 (iv)]{FKZ1} and Lemma \ref{lem2}, we only need to justify
\begin{equation}\label{eq20}
\begin{array}{ll}
& \sup_{f\in C(X)\cap M_\gamma(X)}|\int_X f(x) d\mu- \int_X f(x)d\nu|\vspace{2mm}\\
\leq & \sup_{f \mbox{ has bounded support in } C(X)\cap M_\gamma(X)}|\int_X f(x) d\mu-\int_X f(x) d\nu|
\end{array}
\end{equation}
for any non-negative $\gamma\in\mathbb{R}$ and $\mu, \nu\in\mathcal{\hat{M}}(X)$. To see this, for any small $\epsilon>0$, there exists some $f_\epsilon\in C(X)\cap M_\gamma(X)$, such that
\begin{equation}\label{eq18}
0\leq \sup_{f\in C(X)\cap M_\gamma(X)}\Big|\int_X f(x) d\mu- \int_X f(x)d\nu\Big|-\Big|\int_X f_\epsilon(x) d\mu- \int_X f_\epsilon(x)d\nu\Big|<\epsilon.
\end{equation}
Choose an arbitrary $x_0\in X$, consider the sequence of open balls $\{B(x_0,n)\}_{n\in\mathbb{N}}$ centred at $x_0$  of radius $n\in\mathbb{N}$. Let $n_\epsilon\in\mathbb{N}$ be large enough such that
\begin{equation}\label{eq17}
\mu(B(x_0,n_\epsilon))>\mu(X)-\epsilon \mbox{ and } \nu(B(x_0,n_\epsilon))>\nu(X)-\epsilon.
\end{equation}
The function 
\begin{center}
$f_{1, B(x_0,n_\epsilon)}=\left\{\begin{array}{ll}
1 & x\in B(x_0,n_\epsilon), \vspace{2mm}\\
1-\rho(x, B(x_0,n_\epsilon)) &  0<\rho(x, B(x_0,n_\epsilon))<1, \vspace{2mm}\\
0 &  \rho(x, B(x_0,n_\epsilon))\geq 1.
\end{array}\right.$
\end{center}
is continuous with bounded support. So $f_{1, B(x_0,n_\epsilon)}f_\epsilon\in C(X)\cap M_\gamma(X)$ also has bounded support. Moreover, we have
\begin{equation}\label{eq19}
\begin{array}{ll}
& |\int_X f_{1, B(x_0,n_\epsilon)}f_\epsilon d\mu- \int_X f_{1, B(x_0,n_\epsilon)}f_\epsilon d\nu| \vspace{2mm}\\
\geq & |\int_X f_\epsilon(x) d\mu- \int_X f_\epsilon(x)d\nu|-|\int_X (f_{1, B(x_0,n_\epsilon)}-1)f_\epsilon d\mu- \int_X (f_{1, B(x_0,n_\epsilon)}-1)f_\epsilon d\nu| \vspace{2mm}\\
\geq & |\int_X f_\epsilon(x) d\mu- \int_X f_\epsilon(x)d\nu|-\gamma\mu\big(X\setminus B(x_0,n_\epsilon)\big)-a\nu\big(X\setminus B(x_0,n_\epsilon)\big) \vspace{2mm}\\
\geq & |\int_X f_\epsilon(x) d\mu- \int_X f_\epsilon(x)d\nu|-2\gamma\epsilon \vspace{2mm}\\
\geq & \sup_{f\in C(X)\cap M_\gamma(X)}|\int_X f(x) d\mu- \int_X f(x)d\nu|-(1+2\gamma)\epsilon.
\end{array}
\end{equation} 
The third inequality is due to (\ref{eq17}), while the last one is due to (\ref{eq18}). Now let $\epsilon\rightarrow 0$ in (\ref{eq19}), we get (\ref{eq20}).

\end{proof}

Equipped with all the above results, now we are in a position to prove Theorem \ref{thm3}.

\bigskip
Proof of Theorem \ref{thm3}:

\begin{proof}
The strategy of our proof follows the following diagram.

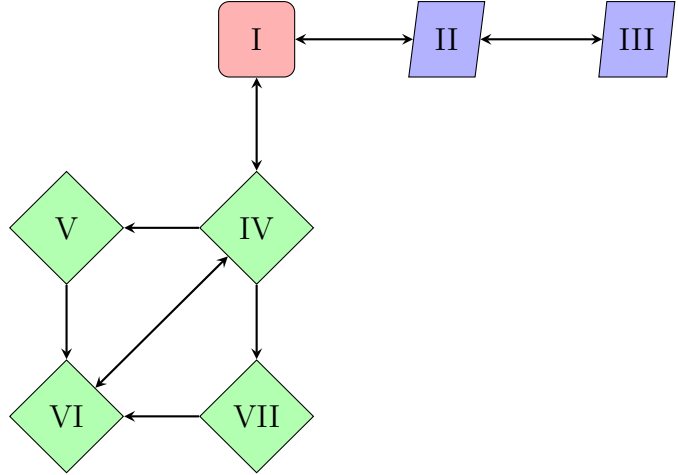
\tikzstyle{startstop} = [rectangle, rounded corners, minimum width=1cm, minimum height=1cm,text centered, draw=black, fill=red!30]

\tikzstyle{io} = [trapezium, trapezium left angle=80, trapezium right angle=100, minimum width=1cm, minimum height=1cm, text centered, draw=black, fill=blue!30, trapezium stretches=true]

\tikzstyle{process} = [rectangle, minimum width=1cm, minimum height=1cm, text centered, draw=black, fill=orange!30]

\tikzstyle{decision} = [diamond, minimum width=1.5cm, minimum height=1.5cm, text centered, draw=black, fill=green!30]

\tikzstyle{arrow1} = [thick,->,>=stealth]
\tikzstyle{arrow2} = [thick, >=triangle 45, <->,>=stealth]

\begin{center}
\begin{tikzpicture}[node distance=2.5cm]
\node(I)[startstop]{I};
\node (II) [io, right of=I] {II};
\node (III) [io, right of=II] {III};
\node (IV) [decision, below of=I] {IV};
\node (VII) [decision, below of=IV] {VII};
\node (V) [decision, left of=IV] {V};
\node (VI) [decision, below of=V] {VI};

\draw [arrow2] (I) -- (II);
\draw [arrow2] (II) -- (III);
\draw [arrow2] (I) -- (IV);
\draw [arrow2] (IV) -- (VI);
\draw [arrow1] (IV) -- (VII);
\draw [arrow1] (VII) -- (VI);
\draw [arrow1] (IV) -- (V);
\draw [arrow1] (V) -- (VI);

\end{tikzpicture}
\end{center}

\begin{itemize}

\item $(I)\Leftrightarrow(II)\Leftrightarrow(III)$: this is due to Lemma \ref{lem1}.

\item $(I)\Leftrightarrow(IV)$: this is due to Lemma \ref{lem2}.

\item $(IV)\Leftrightarrow(VI)$: this is due to Lemma \ref{lem6}.

\item $(IV)\Rightarrow(VII)$: this is trivial.

\item $(VII)\Rightarrow(VI)$: this is because any continuous function on a compact metric space is uniformly continuous.

\item $(IV)\Rightarrow(V)$: this is trivial.

\item $(V)\Rightarrow(VI)$: this is trivial.

\end{itemize}

\end{proof}

\begin{remark}
In all the results such as Lemma \ref{lem1}, \ref{lem6}, Theorem \ref{thm3} and \cite[Theorem 2.5]{FKZ1}, besides the three terms 
\begin{center}
$\sup_{A\in\mathscr{B}}|\mu(A)-\nu(A)|$,
\end{center}
\begin{center}
$\sup_{f\in M_1(X)}|\int_X f(x) d\mu-\int_X f(x) d\nu|$,
\end{center}
\begin{center}
$\sup_{f\in M_\gamma(X)}|\int_X f(x) d\mu- \int_X f(x)d\nu|$,
\end{center}
the \emph{sup} can not be substituted by \emph{Max}, see Example \ref{exm4}.
\end{remark}

The following result follows from Remark \ref{rem3} and Theorem \ref{thm3} instantly.  

\begin{corollary}
For a sequence of measures $\{\nu_n\in\mathcal{\hat{M}}(X)\}_{n=1}^\infty$ and $\nu\in\mathcal{\hat{M}}(X)$ on a $\sigma$-compact metric space $X$, the following conditions are equivalent to each other:
\begin{enumerate}[(I).]
\item $\nu_n\stackrel{TV}{\rightarrow}\nu$ as $n\rightarrow\infty$.
\item $\lim_{n\rightarrow\infty} \sup_{A\mbox{ is compact}}|\nu_n(A)-\nu(A)|=0$.
\item $\lim_{n\rightarrow\infty}\sup_{f \mbox{ has compact support in } M_\gamma(X)}|\int_X f(x) d\nu_n-\int_X f(x) d\nu|=0$ for any non-negative $\gamma\in\mathbb{R}$.
\end{enumerate}
\end{corollary}

The following description of TV sequential convergence of measures follows from the proof of \cite[Theorem 2.5]{FKZ1} essentially.
\begin{corollary}
For a sequence of measures $\{\nu_n\in\mathcal{\hat{M}}(X)\}_{n=1}^\infty$ and $\nu\in\mathcal{\hat{M}}(X)$ on a  metric space $X$,  $\nu_n\stackrel{TV}{\rightarrow}\nu$ as $n\rightarrow\infty$ if and only if
\begin{center}
$\lim_{n\rightarrow\infty}\sup_{f \mbox{ is H\"older in } C_b(X)}|\int_X f(x) d\nu_n-\int_X f(x) d\nu|=0$ 
\end{center}
\end{corollary}
\begin{proof}
This is because the continuous extension $\tilde{f}_{C_1,C_2}(s)$ of $f_{C_1,C_2}(s)$ in the proof of \cite[Theorem 2.5]{FKZ1} can be taken to be H\"older. Details are left to the readers.
\end{proof}

The above results on TV convergence of sequences of measures are set on $\mathcal{\hat{M}}(X)$ with the ambient space being metrizable. While the ambient space $X$ is not metrizable, these results are seldom true in general (some descriptions are not applicable in case $X$ is not metrizable). For example, considering the proof of Theorem \ref{thm5}, we have the following result.

\begin{proposition}
For $K$ being an algebraically closed field, let $X=\mathbb{A}_K^n$ (or $\mathbb{P}_K^n$) be the $n$-dimensional affine (or projective) space over $K$ for some $n\geq 2$. Let $S_n$ (or $S_{n+1}$) be sets of polynomials of $n$ (or $n+1$) variables over $K$ respectively for some $n\geq 2$. Then the affine and projective spaces equipped with the Zariski topology represented by the triples
\begin{center}
$(\mathbb{A}_K^n, S_n, K)$ (or $(\mathbb{P}_K^n, S_{n+1}, K))$
\end{center}
admits a sequence of probability measures $\{\nu_n\in\mathcal{M}(X)\}_{n=1}^\infty$ and  $\nu\in\mathcal{M}(X)$ satisfying both the conditions $(II)(III)$ in Theorem \ref{thm3}, while
\begin{center}
$\nu_n\stackrel{TV}{\nrightarrow}\nu$
\end{center}
as $n\rightarrow\infty$.
\end{proposition}

\section{Attainability of the TV metric between two measures and further sequential convergence in $\mathcal{\hat{M}}(X)$}

Comparing Theorem \ref{thm6}, Example \ref{exm2} and Theorem \ref{thm3}, we can see that there is some fracture between descriptions of these successive modes of convergence. Due to the Hahn decomposition, the total-variation distance of two finite measures can be attained on some bounded measurable function or on some Borel measurable set in (\ref{eq24}). Although  Theorem \ref{thm3} reduces the difficulty on checking TV convergence of sequences of measures in $\mathcal{\hat{M}}(X)$ in some cases, the total-variation metric between two measures may never be attained.

\begin{example}\label{exm4}
Let $X=(0,1)$. Let 
\begin{center}
$\mu=\frac{1}{2}\mathfrak{L}_1|_{(0,1)}+\frac{1}{2}\delta_{\frac{2}{3}}$ 
\end{center}
and 
\begin{center}
$\nu=\mathfrak{L}_1|_{(0,\frac{1}{3})}+\mathfrak{L}_1|_{(\frac{2}{3},1)}+\frac{1}{3}\delta_{\frac{1}{3}}$
\end{center}
be two Borel probabilities on $X$.
\end{example}

One can check that 
\begin{center}
$\Vert\mu-\nu\Vert_{TV}=\frac{2}{3}$
\end{center}
in Example \ref{exm4}. However, there does not exist any open (or closed) set $B$ such that  
\begin{center}
$2|\mu(B)-\nu(B)|=\frac{2}{3}$, 
\end{center}
or any continuous function $f$ such that 
\begin{center}
$|\int_X fd\mu-\int_X fd\nu|=\frac{2}{3}$. 
\end{center}

One should also be careful that although there are many equivalent ways to formulate the total-variation metric between two finite measures on metric spaces (\cite[Theorem 2.5]{FKZ1}, Lemma \ref{lem1}, Lemma \ref{lem6}), the attainability of the total-variation metric between two measures may vary depending on situations. For example, there are examples of unbounded metric spaces $X$ with finite Borel measures $\mu, \nu\in \mathcal{\hat{M}}(X)$ such that    
\begin{equation}\label{eq21}
\Vert\mu-\nu\Vert_{TV}=2|\mu(B)-\nu(B)|
\end{equation}
for some open (or closed) $B\subset X$, while there does not exist any bounded open (or closed) $B\subset X$ satisfying (\ref{eq21}), or any continuous function $f$ such that 
\begin{center}
$|\int_X fd\mu-\int_X fd\nu|=\Vert\mu-\nu\Vert_{TV}$. 
\end{center}

It is possible that any of the four kinds of sequential convergence is still not subtle enough in its applications to real problems. In this case one can define the $F$-convergence and $S$-convergence for 
sequences of measures in $\mathcal{\tilde{M}}(X)$.

\begin{definition}
Let $F$ be a family of functions from $X$ to $\mathbb{R}$.  For a sequence of measures $\{\nu_n\in\mathcal{\tilde{M}}(X)\}_{n=1}^\infty$, we say $\{\nu_n\}_{n=1}^\infty$ $F$-converges to $\nu\in\mathcal{\tilde{M}}(X)$, denoted by $\nu_n\stackrel{F}{\rightarrow}\nu$, if
\begin{center}
$\lim_{n\rightarrow\infty }\int_X f(x) d\nu_n=\int_X f(x) d\nu$
\end{center}
for any  $f\in F$.
\end{definition}

\begin{definition}
Let $S\subset \mathscr{B}$ be a family of sets.  For a sequence of measures $\{\nu_n\in\mathcal{\tilde{M}}(X)\}_{n=1}^\infty$, we say $\{\nu_n\}_{n=1}^\infty$ $S$-converges to $\nu\in\mathcal{\tilde{M}}(X)$, denoted by $\nu_n\stackrel{S}{\rightarrow}\nu$, if
\begin{center}
$\lim_{n\rightarrow\infty } \nu_n(A)=\nu(A)$
\end{center}
for any  $A\in S$.
\end{definition}

Topologies induced from delicate choices of $F$ or $S$ may result in solutions of ones' problems instantly.

\bigskip\bigskip

\small

Address: Department of Mathematics, Binzhou University, Huanghe 5th Road No. 391, Binzhou 256600, Shandong, P. R. China

\bigskip 

Email:  maliangang000@163.com

\end{document}